\documentclass[preprint,12pt]{elsarticle}
\usepackage{amsthm,amsmath,amssymb}
\usepackage{leftidx}
\usepackage[colorlinks=true,citecolor=black,linkcolor=black,urlcolor=blue]{hyperref}
\usepackage{graphicx}
\usepackage{float}
\usepackage{longtable,booktabs}
\usepackage{epstopdf}
\usepackage{epsfig}
\usepackage{subfigure}
\usepackage{makecell}
\usepackage{multirow}

\theoremstyle{plain}
\newtheorem{theorem}{Theorem}
\newtheorem{lemma}[theorem]{Lemma}
\newtheorem{corollary}[theorem]{Corollary}

\theoremstyle{definition}
\newtheorem{definition}[theorem]{Definition}
\newtheorem{example}[theorem]{Example}
\newtheorem{conjecture}[theorem]{Conjecture}

\theoremstyle{remark}
\newtheorem{remark}[theorem]{Remark}

\title{Counterexamples to conjectures by Gross, Mansour and Tucker on partial-dual genus polynomials of ribbon graphs}
\author{Qi Yan~~~Xian'an Jin\footnote{Corresponding author.}\\
\small School of Mathematical Sciences\\[-0.8ex]
\small Xiamen University\\[-0.8ex]
\small P. R. China\\
\small\tt Email:qiyanmath@163.com;xajin@xmu.edu.cn}
\date{}

\begin{document}

\begin{abstract}
Gross, Mansour and Tucker introduced the partial-dual orientable genus polynomial and the partial-dual Euler genus polynomial. They computed these two partial-dual genus polynomials of four families of ribbon graphs, posed some research problems and made some conjectures. In this paper, we introduce the notion of signed sequences of bouquets and obtain the partial-dual Euler genus polynomials for all ribbon graphs with the number of edges less than 4 and
the partial-dual orientable genus polynomials for all orientable ribbon graphs with the number of edges less than 5 in terms of signed sequences. We check all the conjectures and find a counterexample to the Conjecture 3.1 in their paper: There is no orientable ribbon graph having a non-constant partial-dual genus polynomial with only one non-zero coefficient. Motivated by this counterexample, we further find an infinite family of counterexamples to the conjecture. Moreover, we find a counterexample to the Conjecture 5.3 in their paper: The partial-dual Euler-genus polynomial for any non-orientable ribbon graph is interpolating.
\end{abstract}
\maketitle

\section{Introduction}
It is well known that for any ribbon graph \cite{BR2} $G$, equivalently, the old cellularly embedded graph, there is a natural geometric dual ribbon graph $G^{*}$.
In \cite{CG}, Chmutov introduced an extension of geometric duality called partial duality.
Loosely speaking, a partial dual is obtained by forming the geometric dual of a ribbon graph only at a subset of its edges.
Partial duality was used to unify various relations between the Jones polynomials of (virtual) knots and the topological Tutte polynomials of ribbon graphs \cite{Chmu, Chmu2, Lin1}.

In \cite{GMT} Gross, Mansour and Tucker introduced the partial-dual orientable genus polynomials for orientable ribbon graphs and the partial-dual Euler genus polynomials for arbitrary ribbon graphs.

\begin{definition}\label{def-1}\cite{GMT}
The \emph{partial-dual Euler-genus polynomial} (abbr. pDe-polynomial) of any ribbon graph $G$ is the generating function
$$^{\partial}\varepsilon_{G}(z)=\sum_{A\subseteq E(G)}z^{\varepsilon(G^{A})}$$
that enumerates partial duals by Euler-genus.
The \emph{partial-dual orientable genus polynomial} (abbr. pDg-polynomial) of an orientable ribbon graph $G$ is the generating function
$$^{\partial}\Gamma_{G}(z)=\sum_{A\subseteq E(G)}z^{\gamma(G^{A})}$$
that enumerates partial duals by orientable genus.
\end{definition}
Clearly, if $G$ is an orientable ribbon graph, then $^{\partial}\Gamma_{G}(z)=~^{\partial}\varepsilon_{G}(z^{\frac{1}{2}})$.

They gave an edge-contraction/edge-deletion recurrence equation for the subdivision of an edge. This subdivision recursion was then used to
derive closed formulas for the partial-dual orientable genus polynomials of four
families of ribbon graphs. They also posed some research problems and made some conjectures. Conjecture 3.1 and Conjecture 5.3 in their paper state that
\begin{conjecture}\label{con-1}\cite{GMT}
There is no orientable ribbon graph having a non-constant partial-dual genus polynomial with only one non-zero coefficient.
\end{conjecture}

\begin{conjecture}\label{con-2}\cite{GMT}(Interpolating).
The partial-dual Euler-genus polynomial $^{\partial}\varepsilon_{G}(z)$ for any non-orientable ribbon graph is interpolating.
\end{conjecture}

In this paper, we first introduce a new notion of signed sequence for bouquets, i.e. ribbon graphs having a single vertex. Then we obtain the partial-dual Euler genus polynomials for all ribbon graphs with the number of edges less than 4,
the partial-dual orientable genus polynomials for all orientable ribbon graphs with the number of edges less than 5 using these sequences and we find a counterexample to Conjecture \ref{con-1}. Motivated by this counterexample, we further find an infinite family of counterexamples to the conjecture.
Moreover, we find a counterexample to Conjecture \ref{con-2}.

We assume that the readers are familiar with the basic knowledge of topological graph theory and in particular the ribbon graphs, and we refer the readers to \cite{EM, GT}.

\section{Signed sequences of bouquets}

A \emph{signed rotation} of a bouquet is a cyclic ordering of the half-edges at the vertex and if the edge is an untwisted loop, then we give the same sign $+$ to the corresponding two half-edges, and give the different signs (one $+$, the other $-$) otherwise. The sign $+$ is always omitted. See Figure \ref{f01} for an example. Sometimes we will use the signed rotation to represent the bouquet itself.
\begin{figure}[!htbp]
\begin{center}
\includegraphics[width=9cm]{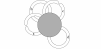}
\caption{A signed rotation of the bouquet is $(a, b, -a, c, b, -c, d, d)$.}
\label{f01}
\end{center}
\end{figure}

\begin{example}\label{ex-2}
Let $\Theta$ be a  non-orientable bouquet with the signed rotation $$(a, b, c, d, -b, -a, c, d).$$
Then $^{\partial}\varepsilon_{\Theta}(z)=4z^{2}+12z^{4}.$
\end{example}

Note that Example \ref{ex-2} is a counterexample to Conjecture \ref{con-2}.

Let $\Theta$ be a bouquet and let $E(\Theta)=\{e_{1}, \cdots, e_{n}\}$ and $e\in E(\Theta)$.
The \emph{interlace number} of $e$, denoted by  $\alpha(e)$, is defined to be
the number of edges which are all interlaced with $e$.
We say that $\beta(e)$ is the \emph{signed interlace number} of $e$, where
\begin{eqnarray*}
\beta(e)=\left\{\begin{array}{ll}
                      \alpha(e), & \mbox{if}~e~\mbox{is an untwisted loop,}\\
                    -\alpha(e), & \mbox{if}~e~\mbox{is a twisted loop.}
                   \end{array}\right.
\end{eqnarray*}

\begin{definition}
The \emph{signed sequence} of the bouquet $\Theta$, denoted by $\mathcal{S}(\Theta)=(\beta(e_{1}), \cdots, \beta(e_{n}))$, is obtained by sorting the signed interlace number from small to large, where $\beta(e_1)\leq \beta(e_2)\leq\cdots\leq \beta(e_n)$.
\end{definition}

Note that $\sum\limits_{1\leq i\leq n}\beta(e_{i})$ is always even.

\begin{example}\label{ex-1}
Let $\Theta$ be a bouquet with the signed rotation $$(a, b, -a, c, b, i, i, d, e, c, f, g, h, d, j, -j, h, -e, g, f).$$
Then $\mathcal{S}(\Theta)=(-4, -1, -0, 0, 1, 2, 2, 2, 3, 5)$ as listed in Table \ref{tab-1}.
\begin{table}
\normalsize
\begin{center}
\begin{tabular}{|c|c|c|c|c|c|c|c|c|c|c|}
\hline
Edges             & $a$ & $b$ & $c$ & $d$ & $e$ & $f$ & $g$ & $h$ & $i$ & $j$     \\ \hline
$\alpha$          & 1   & 2   & 3   & 5   & 4   & 2   & 2   & 1   & 0   & 0       \\ \hline
$\beta$           & -1  & 2   & 3   & 5   & -4  & 2   & 2   & 1   & 0   & -0      \\ \hline
\end{tabular}
\end{center}
\caption{Example \ref{ex-1}}
\label{tab-1}
\end{table}
\end{example}

\section{The partial-dual genus polynomials for ribbon graphs with small number of edges}

In this section we show that signed sequences are sufficient for us to determine the partial-dual genus polynomials of ribbon graphs with small number of edges. Let $P, Q$ be two ribbon graphs, we denote by $P\vee Q$ the \emph{ribbon-join} of $P$ and $Q$. Note that in general the ribbon-join is not unique. We need the following lemmas.

\begin{lemma}\label{lem-02}\cite{GMT}
Let $B=B_{1}\vee B_{2}\vee\cdots \vee B_{k}$. Then
\begin{enumerate}
\item[(1)] $^{\partial}\varepsilon_{B}(z)=~^{\partial}\varepsilon_{B_{1}}(z)~^{\partial}\varepsilon_{B_{2}}(z)\cdots~^{\partial}\varepsilon_{B_{k}}(z)$.
\item[(2)] If $B$ is orientable, then $^{\partial}\Gamma_{B}(z)=~^{\partial}\Gamma_{B_{1}}(z)~^{\partial}\Gamma_{B_{2}}(z)\cdots~^{\partial}\Gamma_{B_{k}}(z)$.
\end{enumerate}
\end{lemma}

\begin{lemma}\label{pro-2}
Let $G$ be a ribbon graph and $A\subseteq E(G)$. Then $^{\partial}\varepsilon_{G}(z)=~^{\partial}\varepsilon_{G^{A}}(z)$.
If $G$ is orientable, then $^{\partial}\Gamma_{G}(z)=~^{\partial}\Gamma_{G^{A}}(z)$.
\end{lemma}

\begin{proof}
This is because the sets of all partial duals of $G$ and $G^A$ are the same.
\end{proof}

It is well known that every connected ribbon graph contains a bouquet in the set of its all partial duals. It suffices for us to consider the partial-dual genus polynomials for bouquets. A loop $e$ at the vertex of a bouquet $\Theta$ is \emph{trivial} if there is no loop of $\Theta$ which interlaces with $e$.

\begin{lemma}\label{pro-1}
Let $B$ be a bouquet and $e\in E(B)$. Then
\begin{eqnarray*}
^{\partial}\varepsilon_{B}(z)=\left\{\begin{array}{ll}
                    2~^{\partial}\varepsilon_{B-e}(z), & \mbox{if}~e~\mbox{is a trivial untwisted loop,}\\
                    2z~^{\partial}\varepsilon_{B-e}(z), & \mbox{if}~e~\mbox{is a trivial twisted loop.}
                   \end{array}\right.
\end{eqnarray*}
\end{lemma}

\begin{proof}
If $e$ is a trivial twisted loop, suppose that the signed rotation of $B$ is $(e, \Phi, -e, \Psi)$.
Then $$B=B_{1}\vee B_{2}\vee B_{3},$$ where signed rotations of $B_{1}$, $B_{2}$ and $B_{3}$ are $(e, -e)$, $(\Phi)$, and $(\Psi)$, respectively.
It follows that $$^{\partial}\varepsilon_{B}(z)=~^{\partial}\varepsilon_{B_{1}}(z)~^{\partial}\varepsilon_{B_{2}}(z)~^{\partial}\varepsilon_{B_{3}}(z)=
2z~^{\partial}\varepsilon_{B_{2}}(z)~^{\partial}\varepsilon_{B_{3}}(z)=2z~^{\partial}\varepsilon_{B-e}(z),$$
by Lemma \ref{lem-02}.
The same reasoning applies to the case when $e$ is a trivial untwisted loop, obtaining $^{\partial}\varepsilon_{B}(z)=2~^{\partial}\varepsilon_{B-e}(z)$.
\end{proof}

\begin{corollary}\label{pro-3}
Let $B$ be a bouquet with $$\mathcal{S}(B)=(\mathcal{P}, \overbrace{-0, \cdots, -0, }^{i}\overbrace{0, \cdots,  0,}^{j}\mathcal{Q}).$$ Then $$^{\partial}\varepsilon_{B}(z)=2^{i+j}z^{i}~^{\partial}\varepsilon_{B-E'}(z),$$
where $E'=\{e\mid\beta(e)=-0~\mbox{or}~0\}.$ Moreover, $$\mathcal{S}(B-E')=(\mathcal{P}, \mathcal{Q}).$$
\end{corollary}

\begin{proof}
For any $e\in E(B)$, if $\beta(e)=-0$, then $e$ is a trivial twisted loop. If $\beta(e)=0$, then $e$ is a trivial untwisted loop.
It follows immediately from Lemma \ref{pro-1}.
\end{proof}

A ribbon graph is called \emph{empty} if it has no edges.  We say that $G$ is \emph{prime}, if there don't exist non-empty ribbon subgraphs $G_{1}, \cdots,  G_{k}$ of $G$
such that $G=G_{1}\vee\cdots \vee G_{k}$ where $k\geq 2$. Otherwise, $G$ is called \emph{non-prime}. Let $e(\Theta)$ be the number of loops of the theta graph $\Theta$.

\begin{theorem}\label{th-02}
Let $\Theta_{1}$ and $\Theta_{2}$ be prime bouquets with $e(\Theta_{i})\leq 3, i\in \{1, 2\}$. Then $\mathcal{S}(\Theta_{1})=\mathcal{S}(\Theta_{2})$ if and only if $\Theta_{1}=\Theta_{2}.$
\end{theorem}

\begin{proof}
%Without loss of generality, we can suppose that $\Theta_{1}$ and $\Theta_{2}$ are alternated bouquets.

There are 15 cases and the proof is straightforward  as shown in the following table.
\begin{longtable}{|c|c|c|c|}
\hline
Cases & \# Edges & Signed sequences & Bouquets \ \ \ \ \\ \hline
{\bf Case 1}  & $e=1$ & (0)         & \includegraphics[width=1.5cm]{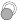} \\ \hline
{\bf Case 2}  & $e=1$ &(-0)         & \includegraphics[width=1.5cm]{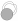} \\ \hline
{\bf Case 3}  & $e=2$ & (1, 1)      & \includegraphics[width=2cm]{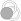} \\ \hline
{\bf Case 4}  & $e=2$ &(-1, 1)      & \includegraphics[width=2cm]{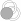} \\ \hline
{\bf Case 5}  & $e=2$ &(-1, -1)     & \includegraphics[width=2cm]{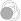} \\ \hline
{\bf Case 6}  & $e=3$ &(1, 1, 2)   & \includegraphics[width=2.5cm]{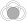} \\ \hline
{\bf Case 7}  & $e=3$ &(2, 2, 2)    & \includegraphics[width=2.5cm]{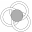} \\ \hline
{\bf Case 8}  & $e=3$ &(-1, 1, 2)   & \includegraphics[width=2.5cm]{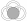} \\ \hline
{\bf Case 9}  & $e=3$ &(-2, 1, 1)   & \includegraphics[width=2.5cm]{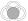} \\ \hline
{\bf Case 10} & $e=3$ &(-2, 2, 2)   & \includegraphics[width=2.5cm]{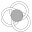} \\ \hline
{\bf Case 11} & $e=3$ &(-1, -1, 2)  & \includegraphics[width=2.5cm]{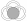} \\ \hline
{\bf Case 12} & $e=3$ &(-2, -1, 1)  & \includegraphics[width=2.5cm]{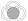} \\ \hline
{\bf Case 13} & $e=3$ &(-2, -2, 2)  & \includegraphics[width=2.5cm]{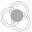} \\ \hline
{\bf Case 14} & $e=3$ &(-2, -1, -1) & \includegraphics[width=2.5cm]{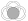} \\ \hline
{\bf Case 15} & $e=3$ &(-2, -2, -2) & \includegraphics[width=2.5cm]{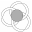} \\ \hline
%\caption{Proof of Theorem \ref{th-02}.}
%\label{tab-3}
\end{longtable}
\end{proof}

\begin{remark}
Theorem \ref{th-02} is not true for non-prime bouquets. For example, suppose that $\Theta_{1}=(a, a, b, b, c, c)$ and $\Theta_{2}=(a, a, b, c, c, b)$.
It is obvious that $\mathcal{S}(\Theta_{1})=\mathcal{S}(\Theta_{2})=(0, 0, 0)$, but $\Theta_{1}\neq\Theta_{2}$.
\end{remark}

\begin{theorem}\label{th-04}
Let $\Theta_{1}$ and $\Theta_{2}$ be bouquets. If $\mathcal{S}(\Theta_{1})=\mathcal{S}(\Theta_{2})$ and $e(\Theta_{i})\leq 3, i\in \{1, 2\}$, then
$$^{\partial}\varepsilon_{\Theta_{1}}(z)=~^{\partial}\varepsilon_{\Theta_{2}}(z).$$
\end{theorem}

\begin{proof}
If $\Theta_{1}$ and $\Theta_{2}$ are prime bouquets, then $\Theta_{1}=\Theta_{2}$ by Theorem \ref{th-02}. Therefore, $$^{\partial}\varepsilon_{\Theta_{1}}(z)=~^{\partial}\varepsilon_{\Theta_{2}}(z).$$
Otherwise, $\Theta_{1}$ and $\Theta_{2}$ are non-prime bouquets. Then there exist non-empty ribbon subgraphs $B_{1}, \cdots,  B_{k}$ of $\Theta_{1}$
such that $\Theta_{1}=B_{1}\vee\cdots \vee B_{k}$ where $B_{j}$ is a prime bouquet for any $1\leq j\leq k$. Since $e(\Theta_{1})\leq 3$, it follows that there are two cases: $k=2$ or $k=3$.

If $k=2$, then $\mathcal{S}(\Theta_{1})\in \{(0, 0), (-0, 0), (-0, -0)$, $(0, 1, 1)$, $(-1, 0, 1),$ $(-1, -1, 0)$, $(-0, 1, 1)$, $(-1, -0, 1), (-1, -1, -0)\}$. We give the proof only for the case $\mathcal{S}(\Theta_{1})=\mathcal{S}(\Theta_{2})=(-1, 0, 1)$. Similar arguments apply to the other cases.
Suppose that $e\in E(\Theta_{1})$ and $e'\in E(\Theta_{2})$ with $\beta(e)=\beta(e')=0$. Then $\mathcal{S}(\Theta_{1}-e)=\mathcal{S}(\Theta_{2}-e')=(-1, 1)$. Note that $\Theta_{1}-e$ and $\Theta_{2}-e'$ are prime bouquets. Thus $\Theta_{1}-e=\Theta_{2}-e'$ by Theorem \ref{th-02}.
We know $^{\partial}\varepsilon_{\Theta_{1}}(z)=2~^{\partial}\varepsilon_{\Theta_{1}-e}(z)$ and
$^{\partial}\varepsilon_{\Theta_{2}}(z)=2~^{\partial}\varepsilon_{\Theta_{2}-e'}(z)$ by Corollary \ref{pro-3}.
Therefore, $^{\partial}\varepsilon_{\Theta_{1}}(z)=~^{\partial}\varepsilon_{\Theta_{2}}(z).$

If $k=3$, then $\mathcal{S}(\Theta_{1})\in \{(0, 0, 0), (-0, 0, 0), (-0, -0, 0), (-0, -0, -0)\}.$
For any case, the proof is immediately by Corollary \ref{pro-3}.
\end{proof}

\begin{remark}
If $e(\Theta_{i})\geq 4$, Theorem \ref{th-04} is sometimes wrong. For example, $$\Theta_{1}=(a, c, -a, d, b, d, c, -b),$$ and $$\Theta_{2}=(a, c, b, -a, -b, d, c, d).$$ Then $\mathcal{S}({\Theta_{1}})=\mathcal{S}({\Theta_{2}})=(-2, -1, 1, 2),$ but $$^{\partial}\varepsilon_{\Theta_{1}}(z)=4z^{2}+8z^{3}+4z^{4},$$ $$^{\partial}\varepsilon_{\Theta_{2}}(z)=2z+2z^{2}+8z^{3}+4z^{4}.$$
\end{remark}

\begin{theorem}\label{th-03}
Let $\Theta_{1}$ and $\Theta_{2}$ be orientable bouquets with $e(\Theta_{i})\leq 4, i\in \{1, 2\}$. Then $\mathcal{S}(\Theta_{1})=\mathcal{S}(\Theta_{2})$  if and only if $\Theta_{1}=\Theta_{2}.$
\end{theorem}

\begin{proof}
We give the proof only for the case $e(\Theta_{1})=e(\Theta_{2})=4$, the other cases have been discussed in Theorem \ref{th-02}.
There are 6 cases and the proof is straightforward as shown in the following table.

\begin{longtable}{|c|c|c|}
\hline
Cases &  Signed sequences & Bouquets\\ \hline
{\bf Case 1}   & (1, 1, 1, 3)         & \includegraphics[width=2.5cm]{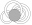} \\ \hline
{\bf Case 2}   & (1, 1, 2, 2)         & \includegraphics[width=2.5cm]{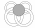} \\ \hline
{\bf Case 3}   & (1, 2, 2, 3)         & \includegraphics[width=2.5cm]{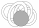} \\ \hline
{\bf Case 4}   & (2, 2, 2, 2)         & \includegraphics[width=2.5cm]{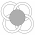} \\ \hline
{\bf Case 5}   & (2, 2, 3, 3)         & \includegraphics[width=2.5cm]{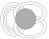} \\ \hline
{\bf Case 6}   & (3, 3, 3, 3)         & \includegraphics[width=2.5cm]{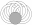} \\ \hline
\end{longtable}
\end{proof}

\begin{remark}
~

\begin{itemize}
  \item If $e(\Theta_{i})\geq 5$, Theorem \ref{th-03} is sometimes wrong. For example, let $$\Theta_{1}=(a, b, c, a, d, c, e, b, d, e)$$ and $$\Theta_{2}=(a, b, c, a, d, e, c, b, e, d).$$ Then $\mathcal{S}({\Theta_{1}})=\mathcal{S}({\Theta_{2}})=(2, 2, 2, 3, 3),$ but numbers of faces
      $$f({\Theta_{1}})=2, f({\Theta_{2}})=4.$$ Thus  $\Theta_{1}\neq\Theta_{2}.$
  \item If $\Theta_{1}$ and $\Theta_{2}$ are non-orientable bouquets, Theorem \ref{th-03} is sometimes wrong. For example, let $$\Theta_{1}=(a, b, -a, c, -b, d, c, d),$$ and $$\Theta_{2}=(a, b, c, -b, d, -a, d, c).$$ It is obvious that  $\mathcal{S}({\Theta_{1}})=\mathcal{S}({\Theta_{2}})=(-2, -1, 1, 2),$ but
      $$f({\Theta_{1}})=2, f({\Theta_{2}})=1.$$ Therefore, $\Theta_{1}\neq\Theta_{2}.$

\end{itemize}
\end{remark}

\begin{theorem}\label{th-05}
Let $\Theta_{1}$ and $\Theta_{2}$ be orientable bouquets. If $\mathcal{S}(\Theta_{1})=\mathcal{S}(\Theta_{2})$ and $e(\Theta_{i})\leq 4, i\in \{1, 2\}$, then $$^{\partial}\Gamma_{\Theta_{1}}(z)=~^{\partial}\Gamma_{\Theta_{2}}(z).$$
\end{theorem}

\begin{proof}
If $\Theta_{1}$ and $\Theta_{2}$ are prime bouquets, then $\Theta_{1}=\Theta_{2}$ by Theorem \ref{th-03}. Therefore, $$^{\partial}\Gamma_{\Theta_{1}}(z)=~^{\partial}\Gamma_{\Theta_{2}}(z).$$
Otherwise, $\Theta_{1}$ and $\Theta_{2}$ are non-prime bouquets.
We give the proof only for the case $e(\Theta_{1})=e(\Theta_{2})=4$, the other cases have been discussed in Theorem \ref{th-04}.
Then there exist non-empty ribbon subgraphs $B_{1}, \cdots,  B_{k}$ of $\Theta_{1}$
such that $\Theta_{1}=B_{1}\vee\cdots \vee B_{k}$ where $B_{j}$ is a prime bouquet for any $1\leq j\leq k$.
There are two cases.

\noindent{\bf Case 1} If $k\geq3$, then there exist $e\in \Theta_{1}$ and $e'\in \Theta_{2}$ such that $\beta(e)=\beta(e')=0$.
Note that $\mathcal{S}(\Theta_{1}-e)=\mathcal{S}(\Theta_{2}-e')$  and
$^{\partial}\Gamma_{\Theta_{1}}(z)=2~^{\partial}\Gamma_{\Theta_{1}-e}(z),$
$^{\partial}\Gamma_{\Theta_{2}}(z)=2~^{\partial}\Gamma_{\Theta_{2}-e'}(z)$ by Corollary \ref{pro-3}.
Since $e(\Theta_{1}-e)=e(\Theta_{2}-e')=3,$ it follows that $^{\partial}\Gamma_{\Theta_{1}-e}(z)=~^{\partial}\Gamma_{\Theta_{2}-e'}(z)$ by Theorem \ref{th-04}.
Thus $^{\partial}\Gamma_{\Theta_{1}}(z)=~^{\partial}\Gamma_{\Theta_{2}}(z).$

\noindent{\bf Case 2} If $k=2$, then $\mathcal{S}(\Theta_{1})\in \{(1, 1, 1, 1), (0, 1, 1, 2), (0, 2, 2, 2)\}$.
If $\mathcal{S}(\Theta_{1})=\{1, 1, 1, 1\}$, it is evident that $\Theta_{1}=\Theta_{2}$ as shown in Figure \ref{f23}.
Therefore, $^{\partial}\Gamma_{\Theta_{1}}(z)=~^{\partial}\Gamma_{\Theta_{2}}(z).$
If $\mathcal{S}(\Theta_{1})=(0, 1, 1, 2)$ or $\mathcal{S}(\Theta_{1})=(0, 2, 2, 2)$, this follows by the same method as in Case 1.
\begin{figure}[!htbp]
\centering
\includegraphics[width=8cm]{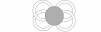}
\caption{A bouquet with signed sequence $(1, 1, 1, 1)$.}
\label{f23}
\end{figure}

\end{proof}

\begin{remark}
If $e(\Theta_{i})\geq 5$, Theorem \ref{th-05} is sometimes wrong. For example, let $$\Theta_{1}=(a, b, a, c, b, d, e, c, d, e),$$ and $$\Theta_{2}=(a, b, a, c, d, b, e, d, c, e).$$ Then  $\mathcal{S}({\Theta_{1}})=\mathcal{S}({\Theta_{2}})=(1, 2, 2, 2, 3),$ but $$^{\partial}\Gamma_{\Theta_{1}}(z)=12z+20z^{2},$$ $$^{\partial}\Gamma_{\Theta_{2}}(z)=2+14z+16z^{2}.$$
\end{remark}

By Theorems \ref{th-04} and \ref{th-05}, we give the
pDe-polynomial for an arbitrary prime ribbon graph $\Theta$ with $e(\Theta)\leq 3$
and give the pDg-polynomial for an arbitrary prime ribbon graph $\Theta'$ with $e(\Theta')\leq 4$ in terms of signed sequences of bouquets as shown in Table \ref{tab-2}.

The bouquet with signed sequence $(2, 2, 2)$ is a counterexample to Conjecture \ref{con-1}.

\begin{table}
\normalsize
\begin{center}
\begin{tabular}{|c|c|c|}
\hline
$\mathcal{S}(\Theta)$ & $^{\partial}\varepsilon_{\Theta}(z)$   & $^{\partial}\Gamma_{\Theta}(z)$\\ \hline
(0)                                       & $2$               & $2$          \\ \hline
(-0)                                      & $2z$              & $\diagup$    \\ \hline
(1, 1)                                    & $2+2z^{2}$        &$2+2z$        \\ \hline
(-1, 1), (-1, -1)                       & $2z+2z^{2}$       & $\diagup$    \\ \hline
(1, 1, 2)                                 & $2+6z^{2}$        & $2+6z$       \\ \hline
(2, 2, 2)                                 & $8z^{2}$          &$8z$          \\ \hline
(-1, 1, 2), (-2, -1, 1), (-2, -2, 2)  & $2z+2z^{2}+4z^{3}$& $\diagup$    \\ \hline
(-2, 1, 1), (-2, -2, -2)                & $2z+6z^{2}$       & $\diagup$    \\ \hline
(-1, -1, 2), (-2, -1, -1), (-2, 2, 2) &$4z^{2}+4z^{3}$    & $\diagup$    \\ \hline
(1, 1, 1, 3)                              & $2+14z^{2}$       & $2+14z$      \\ \hline
(1, 1, 2, 2), (2, 2, 2, 2)             &$2+10z^{2}+4z^{4}$ &$2+10z+4z^{2}$\\ \hline
(1, 2, 2, 3), (2, 2, 3, 3)              &$12z^{2}+4z^{4}$   & $12z+4z^{2}$ \\ \hline
(3, 3, 3, 3)                              &$8z^{2}+8z^{4}$    & $8z+8z^{2}$  \\
\hline

\end{tabular}
\end{center}
\caption{The partial-dual genus polynomials for prime bouquets $\Theta$ with $e(\Theta)\leq 4$.}
\label{tab-2}
\end{table}

We now present two small examples.

\begin{example}
Let $\Theta$ be a bouquet with signed rotation $$(h, a, b, c, d, c, a, d, b, h, i, e, f, -e, g, -f, g, -i).$$
Since $\beta(h)=0, \beta(i)=-0$, by Corollary \ref{pro-3}$$^{\partial}\varepsilon_{\Theta}(z)=2^{2}z~^{\partial}\varepsilon_{\Theta-\{h, i\}}(z).$$
Moreover, $\Theta-\{h, i\}=(a, b, c, d, c, a, d, b, e, f, -e, g, -f, g)=\Theta_{1}\vee\Theta_{2}$, where
$\Theta_{1}=(a, b, c, d, c, a, d, b), \Theta_{2}=(e, f, -e, g, -f, g)$. Since $\mathcal{S}(\Theta_{1})=(1, 1, 2, 2)$,
$\mathcal{S}(\Theta_{2})=(-2, -1, 1)$, it follows that $^{\partial}\varepsilon_{\Theta_{1}}(z)=2+10z^{2}+4z^{4}$,
$^{\partial}\varepsilon_{\Theta_{2}}(z)=2z+2z^{2}+4z^{3}$ by Table \ref{tab-2}. Then
\begin{eqnarray*}
^{\partial}\varepsilon_{\Theta}(z)&=& 2^{2}z~^{\partial}\varepsilon_{\Theta-\{h, i\}}(z)=2^{2}z~^{\partial}\varepsilon_{\Theta_{1}}(z)~^{\partial}\varepsilon_{\Theta_{2}}(z)\\
&=& 2^{2}z(2+10z^{2}+4z^{4})(2z+2z^{2}+4z^{3})\\
&=& 16z^{2}+16z^{3}+112z^{4}+80z^{5}+192z^{6}+32z^{7}+64z^{8}.
\end{eqnarray*}
\end{example}
\begin{example}
Let $G$ be a ribbon graph as shown in Figure \ref{f24} and let $A=\{a, b, c, d, e\}$. Then $G^{A}$ is a bouquet.
An easy computation shows that the signed rotation of $G^{A}$ is $$(a, c, h, c, b, h, b, a, d, g, e, f, e, d, g, f).$$
Since $\beta(a)=0$, it follows that
$$^{\partial}\Gamma_{G^{A}}(z)=2~^{\partial}\Gamma_{G^{A}-a}(z).$$
Furthermore, $G^{A}-a=(c, h, c, b, h, b, d, g, e, f, e, d, g, f)=\Theta_{1}\vee\Theta_{2}$, where
$\Theta_{1}=(c, h, c, b, h, b), \Theta_{2}=(d, g, e, f, e, d, g, f)$.
Since $\mathcal{S}(\Theta_{1})=(1, 1, 2)$,
$\mathcal{S}(\Theta_{2})=(1, 2, 2, 3)$, we have $^{\partial}\Gamma_{\Theta_{1}}(z)=2+6z$,
$^{\partial}\Gamma_{\Theta_{2}}(z)=12z+4z^{2}$ by Table \ref{tab-2}. Therefore,
\begin{eqnarray*}
^{\partial}\Gamma_{G}(z)&=& ^{\partial}\Gamma_{G^{A}}(z)=2~^{\partial}\Gamma_{G^{A}-a}(z)=2~^{\partial}\Gamma_{\Theta_{1}}(z)~^{\partial}\Gamma_{\Theta_{2}}(z)\\
&=& 2(2+6z)(12z+4z^{2})=48z+160z^{2}+48z^{3}.
\end{eqnarray*}
\begin{figure}[!htbp]
\centering
\includegraphics[width=10cm]{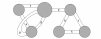}
\caption{A ribbon graph $G$.}
\label{f24}
\end{figure}
\end{example}

\begin{remark}
Note that for any subset $A\subseteq E(G)$, $G^{A}$ and $G^{A^{c}}$ are geometric duals, having the same Euler genus or orientable genus. Thus each term of the partial dual genus polynomial of a non-empty ribbon graph has an even coefficient.
\end{remark}

\section{An infinite family of counterexamples}

In this section, we further give an infinite family of counterexamples to Conjecture \ref{con-1}.

\begin{lemma}\label{lem-03}
Let $\Theta_{t}$ be a bouquet with the signed rotation $$(1, 2, 3, \cdots, t, 1, 2, 3, \cdots, t).$$ Then
\begin{eqnarray*}
\gamma(\Theta_{t})=\left\{\begin{array}{ll}
                    \frac{1}{2}(t-1), & \mbox{if}~t~\mbox{is odd,}\\
                    \frac{1}{2}t, & \mbox{if}~t~\mbox{is even.}
                   \end{array}\right.
\end{eqnarray*}
\end{lemma}

\begin{proof}
The result is easily verified when $t=1, 2$, Now let $t\geq 3$. By Figure \ref{f25}, we have
\begin{figure}[htbp]
\begin{center}
\includegraphics[width=13cm]{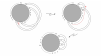}
\caption{Proof of Lemma \ref{lem-03}.}
\label{f25}
\end{center}
\end{figure}
\begin{eqnarray*}
&&f(1, 2, 3, \cdots,t-1, t, 1, 2, 3, \cdots, t-1, t)\\
&=& f(1, \cdots, t-2, 1, \cdots, t-2, t-1, t, t-1, t)\\
&=& f(1, \cdots, t-2, 1, \cdots, t-2).
\end{eqnarray*}
In the same manner we can see that if $t$ is odd, then
\begin{eqnarray*}
f(1, 2, 3, \cdots, t, 1, 2, 3, \cdots, t)&=&f(1, \cdots, t-2, 1, \cdots, t-2)\\
&=& \cdots=f(1, 1)=2.
\end{eqnarray*}
It is easy to check that $$\gamma(1, 2, 3, \cdots, t, 1, 2, 3, \cdots, t)=\frac{1}{2}(t-1).$$
If $t$ is even, then
\begin{eqnarray*}
f(1, 2, 3, \cdots, t, 1, 2, 3, \cdots, t)&=&f(1, \cdots, t-2, 1, \cdots, t-2)\\
&=& \cdots=f(1, 2, 1, 2)=1.
\end{eqnarray*}
Thus, $$\gamma(1, 2, 3, \cdots, t, 1, 2, 3, \cdots, t)=\frac{1}{2}t.$$
\end{proof}

\begin{lemma}\cite{GMT} \label{lem-01}
Let $G$ be a bouquet and $A\subseteq E(G)$. Then
\begin{enumerate}
\item[(1)] $\varepsilon(G^{A})=\varepsilon(A)+\varepsilon(A^{c})$.
\item[(2)] If $G$ is orientable, then $\gamma(G^{A})=\gamma(A)+\gamma(A^{c})$.
\end{enumerate}
\end{lemma}

\begin{theorem}\label{th-01}
\begin{eqnarray*}
^{\partial}\Gamma_{\Theta_{t}}(z)=\left\{\begin{array}{ll}
                    2^{t}z^{\frac{1}{2}(t-1)}, & \mbox{if}~t~\mbox{is odd,}\\
                    2^{t-1}z^{\frac{1}{2}t}+2^{t-1}z^{\frac{1}{2}(t-2)}, & \mbox{if}~t~\mbox{is even.}
                   \end{array}\right.
\end{eqnarray*}
\end{theorem}

\begin{proof}
For any subset $A\subseteq E(\Theta_{t})$, it is easy to calculate the genus of $\Theta_{t}^{A}$  by Lemmas \ref{lem-03} and \ref{lem-01}.

If $t$ is odd, then $|A|$ and $|A^{c}|$ are exactly one odd and one even. Therefore
$$\gamma(\Theta_{t}^{A})=\gamma(A)+\gamma(A^{c})=\frac{1}{2}(t-1).$$ Hence $$^{\partial}\Gamma_{\Theta_{t}}(z)=2^{t}z^{\frac{1}{2}(t-1)}.$$

If $t$ is even, then $|A|$ and $|A^{c}|$ are both odd or both even. It follows that
\begin{eqnarray*}
\gamma(\Theta_{t}^{A})=\gamma(A)+\gamma(A^{c})=\left\{\begin{array}{ll}
                    \frac{1}{2}(t-2), & \mbox{if}~|A|~\mbox{is odd,}\\
                    \frac{1}{2}t, & \mbox{if}~|A|~\mbox{is even.}
                   \end{array}\right.
\end{eqnarray*}
Thus $$^{\partial}\Gamma_{\Theta_{t}}(z)=2^{t-1}z^{\frac{1}{2}t}+2^{t-1}z^{\frac{1}{2}(t-2)}.$$
\end{proof}

In the cast $t=3$, it is exactly the bouquet with signed sequence $(2,2,2)$.

%\section*{References}
%\bibliographystyle{model1b-num-names}
%\bibliographystyle{amcjoucc}
%\bibliography{amcexample}

\end{document}